\documentclass[10pt]{article}
\usepackage[usenames]{color}
\usepackage{xcolor}
\usepackage{hyperref}
\usepackage[utf8]{inputenc}
\usepackage{amsmath,amsfonts,amscd,amssymb,amsthm,latexsym,eucal,graphics,amsbsy,romannum}

\providecommand{\bbN}{{\mathbb N}} 

\providecommand{\Syl}{{\mathrm{Syl}}} 
 
\providecommand{\N}{{\mathrm{N}}} 
\providecommand{\C}{{\mathrm{C}}}
\providecommand{\Z}{{\mathrm{Z}}}

\providecommand{\rmO}{{\mathrm{O}}}
\providecommand{\F}{{\mathrm{F}}} 

\providecommand{\teq}{{\ \trianglelefteq \ }}

\providecommand{\gen}[1]{\langle #1 \rangle}

\providecommand{\gen}[1]{\langle #1 \rangle}
\providecommand{\gen}[1]{\langle #1 \rangle}

\providecommand{\wtG}{{\widetilde{G}}}
\providecommand{\Ind}{{\operatorname{Ind}}}

\providecommand{\bbN}{\mathbb{N}}

\newtheorem{lem}{Lemma}
\newtheorem*{theorem}{Theorem}

\newtheorem*{cor}{Corollary}

\newtheorem{question}{Question}

\title{On $A$-Groups with the Same Index Set as a Nilpotent Group}
\author{
Wei Zhou\thanks{Sobolev Institute of Mathematics, Novosibirsk 630090, Russia.
Email: zhouwyx@outlook.com.} 
\footnotemark[4]
,
Ilya Gorshkov\thanks{School of Mathematical Sciences, Hebei Key Laboratory of Computational Mathematics and Applications, Hebei Normal University, Shijiazhuang 050024, P. R. China; Novosibirsk State Technical University, Novosibirsk 630073 Russia. Email: ilygor8@gmail.com.}
\thanks{The author was supported by the grant of The Natural Science Foundation of Hebei Province (Project No. A2023205045).} 
\footnotemark[4]
}

\date{}

\begin{document}

\pagenumbering{arabic}

\maketitle

\footnotetext[4]{The work is supported by the Mathematical Center in Akademgorodok under the agreement No. 075-15-2025-348 with the Ministry of Science and Higher Education of the Russian Federation.}

\noindent\textbf{Abstract.} 
    Let $G$ be a finite group and 
    $\N(G)$ be the set of conjugacy class sizes of $G$.
    For a prime $p$, 
    let $|G||_p$ be the highest $p$-power dividing some element of $\N(G)$.
    and define $|G||=\Pi_{p\in \pi (G)}|G||_p$.
    $G$ is said to be an $A$-group if all its Sylow subgroups are abelian.
    We prove that if $G$ is an $A$-group
    such that 
    $\N(G)$ contains $|G||_p$ for every $p\in \pi(G)$ as well as $|G||$,
    then $G$ must be abelian.
    This result gives a positive answer to a question 
    posed by Camina and Camina in 2006.



\section*{Introduction}

In this paper, all groups considered are finite, 
and $p$ always denotes a prime number.
Let $x$ be an element of a group $G$.
The index of $x$ in $G$ is given by $|G:\C_G(x)|$ and is denoted by $\Ind(G,x)$.
Note that $\Ind(G,x)$ is the size of the conjugacy class of $x$.
Define $\N(G)$ as the set of conjugacy class sizes of all elements in $G$,
i.e., $\N(G)=\{ \Ind(G,x) \mid x\in G \}$.
For convenience, we refer to $\N(G)$ as the index set of $G$.
Let $|G||_p$ denote the highest $p$-power $p^n$ which divides some element of $\N(G)$.
Furthermore, we write $\pi(G)$ for the set of all prime divisors of $|G|$,
and define $|G|| = \Pi_{p\in \pi(G)}|G||_p$.
Note that $|G||_p$ and $|G||$ are not necessarily contained in $\N(G)$.

Many researchers have been interested in using $\N(G)$ to characterize the structure of $G$. 
It\^o proved that if $\N(G) = \{1, n\}$, 
then $G$ must be the direct product of a $p$-group and an abelian $p'$-group \cite{Ito1953}. 
Ishikawa further showed that the nilpotency class of such groups is at most 3 \cite{Ishikawa2002}. 
More related results can be found in \cite{Camina2011}.

Let $\pi(G)=\{p_1, p_2,\ldots, p_k\}$.
It is easy to verify that if $G$ is nilpotent, 
then $\N(G) = \N(P_1) \times \N(P_2) \times \ldots \times \N(P_k)$, 
where $P_i\in \Syl_{p_i}(G)$, $1\leq i \leq k$.
Note that each $\N(P_i)$ consists of powers of $p_i$, including $1$. 
In \cite{Cossey2000}, Cossey proved that 
every finite set of $p$-powers containing $1$
can be an index set of some $p$-group.
Therefore, for any set $\Omega=\Omega_1 \times \Omega_2 \times \dots \times \Omega_r$,
where each $\Omega_i$ is a finite set of $p_i$-powers containing $1$, and $p_1, p_2, \dots, p_r$ are distinct primes, there exists a nilpotent group $G$ such that $N(G)=\Omega$.
It is natural to ask: 

\begin{question}[{\cite[Question 1]{Camina2006}}]
\label{Q1}
    If a group $G$ has the same index set as a nilpotent group, 
    that is,  
    $N(G) = \Omega_1 \times \Omega_2 \times \dots \times \Omega_r$,
    where each $\Omega_i$ is a finite set of $p_i$-powers containing $1$, 
    and $p_1, p_2, \dots, p_r$ are distinct primes, 
    does it follow that $G$ is nilpotent?
\end{question}

    It is not true in general,
    as some counterexamples are provided in \cite{Camina2006}.
    However, Question $\ref{Q1}$ has a positive answer in some special cases.
    For example,  
    if $\N(G)=\{1,p_1^{m_1}\}\times \{1,p_2^{m_2}\}\times \cdots \times \{1,p_k^{m_k}\}$,
    where $p_1^{m_1},p_2^{m_2},\ldots,p_k^{m_k}$ are powers of distinct primes,
    then $G$ is nilpotent \cite{Casolo2012}.
    More generally, 
    if $\N(G) = \{1, n_1\} \times \{1, n_2\} \times \cdots \times \{1, n_r\}$, 
    where $n_1, n_2, \ldots, n_r$ are pairwise coprime integers, 
    then $G$ is nilpotent \cite{Gorshkov2025}.
    Recall that a finite group is called an $A$-group
    if all its Sylow subgroups are abelian.
    Camina and Camina posed the following question and conjectured that the answer should be positive.  
    
\begin{question}[{\cite[Question 2]{Camina2006}}]
\label{Q2}
    Let $G$ and $H$ be finite groups 
    with $H$ nilpotent and $G$ an $A$-group.
    If $\N(G)=\N(H)$, must $G$ be nilpotent (and hence abelian)?
\end{question}

    Camina and Camina proved that if $G$ is additionally metabelian 
    (i.e., its commutator subgroup $G'$ is abelian), 
    then $G$ must be abelian \cite{Camina2006}.
    In this paper, we will completely resolve Question \ref{Q2}.
    In fact, we prove the following more general result.

\begin{theorem}
    Let $G$ be an $A$-group such that $\N(G)$ contains $|G||_p$ for all $p\in \pi(G)$ and $|G||$. Then $G$ is an abelian group.
\end{theorem}

    Since every nilpotent group satisfies the conditions of the theorem, 
    we can immediately get the following corollary,
    which provides a positive answer to Question \ref{Q2}.

\begin{cor}
    Let $G$ be an $A$-group 
    that shares the same index set with some nilpotent group. 
    Then $G$ must be abelian.
\end{cor}

\section{Preliminaries}

\begin{lem}[\cite{GorA2}, Lemma 1.4]
\label{basic}
    For a finite group $G$, take $K\teq G$ and $x\in G$.
    The following claims hold:

    (i) $\Ind (K, x)$ and $\Ind(G/K, xK)$ divide $\Ind(G,x)$.
    
    (ii) If $y\in G$ with $xy=yx$ and $(|x|,|y|)=1$, 
    then $\C_G(xy)=\C_G(x)\cap \C_G(y)$.

    (iii) $\C_G(x)K/K\leq \C_{G/K}(xK)$.

    (iv) If $(|x|, |K|)=1$, then $\C_G(x)K/K = \C_{G/K}(xK)$. 
\end{lem}

\begin{lem}[\cite{Camina2006}, Lemma 2]
\label{CL2}
    Let $A$ be an abelian group acting faithfully on an abelian group $V$
    (written additively) of coprime order.
    Then $V$ has a regular orbit.
\end{lem}

\begin{lem}[{\cite[Theorem 5.2.3]{Gorenstein}}]
\label{Go}
    Let $A$ be a $p'$-group of automorphisms of the abelian group $P$. 
    Then $P=\C_P(A)\times[P,A]$.
\end{lem}

\begin{lem}
\label{centre}
    Let $H$ be an abelian normal subgroup of $G$ and $g\in \C_G(H)$.
    If $\C_G(g)/H = \C_{G/H}(gH)$, 
    then for any $h\in H$, $\C_G(hg) = \C_G(h)\cap \C_G(g)$.
\end{lem}

\begin{proof}
    It is obvious that $ \C_G(h)\cap \C_G(g) \subseteq \C_G(hg)$.
    We will show that $\C_G(hg)\subseteq  \C_G(h)\cap \C_G(g)$.
    Let $x\in \C_G(hg)$.
    We have $(hg)^x = h^x g^x = hg$.
    It follows that $g^x = (h^{-1})^x h g \in gH$.
    Hence $xH \in \C_{G/H}(gH)$.
    Since $\C_G(g)/H = \C_{G/H}(gH)$, we have $x\in \C_G(g)$.
    Therefore $h^x= h$, i.e. $x\in \C_G(h)$,
    which implies that $x\in \C_G(h)\cap \C_G(g)$.
    Thus, $\C_G(hg)\subseteq  \C_G(h)\cap \C_G(g)$.
\end{proof}

\begin{lem}
\label{size}
    Suppose that $G$ is a finite group 
    and $H$ is an abelian normal $p$-subgroup of $G$.
    Let $g$ be a $p'$-element of $G$ and $h\in H$.
    If $|hg|=|g|$, then $hg = g^y$ where $y\in H$.
    Moreover, $|\C_G(hg)|=|\C_G(g)|$.
\end{lem}

\begin{proof}
    By Lemma \ref{Go}, $H= \C_H(g)\times [H, \gen{g}]$.
    Let $h=uv$ where $u\in \C_H(g)$, $v\in [H,\gen{g}]$ 
    and $n=|g|=|hg|$.
    Here $n$ is a $p'$-number.
    We have $(hg)^n = hh^{g^{-1}}\ldots h^{g^{1-n}} g^n
    =hh^{g^{-1}}\ldots h^{g^{1-n}} 
    =u^n vv^{g^{-1}}\ldots v^{g^{1-n}}= 1$.
    Therefore $u=1$ and so $h\in [H,\gen{g}]$.
    Since $H$ is abelian and normal in $G$, 
    we have $[a,g][b,g] = [ab,g]$ and
    $[a,g^t]=[a,g][a^g, g]\ldots [a^{g^{t-1}}, g]=[aa^g\ldots a^{g^{t-1}}, g]$,
    where $a,b\in H$ and $t\in \bbN$.
    Hence $[H,\gen{g}]=\{[x, g]|x\in H\}=[H, g]$.
    Similarly, $[H,\gen{g}]=[H,g^{-1}]$
    and we have $h=[y,g^{-1}] = y^{-1} g y g^{-1}$, where $y\in H$.
    Hence $hg = g^y$ and  $\C_G(hg) = \C_G(g^{y}) = (\C_G(g))^{y}$.
    Therefore $|\C_G(hg)|=|\C_G(g)|$.
\end{proof}

\begin{lem}
\label{L4}
    Let $P\in \Syl_p(G)$ be abelian, 
    $H$ be a normal $p$-subgroup of $G$. 
    Let $g\in P\setminus H$ and $T\leq G$ such that $T/H= \C_{G/H}(gH)$.
    If $\C_G(g) < T$, 
    then $g = xy$, where $1\neq x\in \C_G(T)$ and $1\neq y\in H$.
    Moreover, $\C_G(x)=T$, $\C_G(x)/H = \C_{G/N}(xH)$ 
    and $\C_G(y)\cap T=\C_G(g)$.
\end{lem}

\begin{proof}
    Let $t\in T$. We have $g^t = gh$ with $h\in H$.
    Therefore, $\gen{g}H$ is normal in $T$.
    Let $K= \gen{g}H$ and $L= \C_T(K)$.
    Since Sylow $p$-subgroups of $G$ are abelian,
    $T/L$ is a $p'$-group.
    Consider $M=K\rtimes T/L$, where $k^{tL}=k^t$.
    Since $K$ is normal in $T$ and $L = \C_T(K)$, $M$ is well-defined.
    In $M$, by Lemma \ref{Go}, we have $K = \C_K(T/L) \times [K, T/L]$.
    Therefore in $G$, $K=\C_K(T)\times [K,T]$.
    We have $[K,T]H/H = [K/H, T/H] = [\gen{g}H/H, \C_{G/H}(gH)]= 1$.
    Hence $[K,T]\leq H$.
    Therefore $\C_K(T)>1$.
    Since $g\notin \C_K(T)$ and $g\notin H$, 
    we have $g=xy$, where $x\in \C_K(T)$ and $y\in [K,T]\leq H$.
    
    Since $gy^{-1}\in \C_K(T)$,
    we have $T\leq \C_G(gy^{-1})$ and so $\C_G(g)\leq \C_G(gy^{-1})$.
    Thus, $\C_G(g)\leq \C_G(y^{-1})=\C_G(y)$.
    For any element $s\in T\setminus \C_G(g)$,
    if $s\in \C_G(y)$, then $gx^{-1}s= sgx^{-1}$.
    However, $x\in \C_K(T)$.
    Therefore $gs=sg$, a contradiction.
    So $\C_G(y) \cap T = \C_G(g)$.

    Since $x\in \C_K(T)$, $T\leq \C_G(x)$.
    Let $c\in \C_G(x)$.
    Then $(gy^{-1})^c=gy^{-1}$ and so $g^{-1} g^c =y^{-1} y^c\in H$.
    Hence $c\in T$.
    Therefore $\C_G(x)=T$ and $\C_G(x)/H = \C_{G/N}(gH) = \C_{G/N}(xH)$.
\end{proof}

Let $H$ be a normal abelian subgroup of $G$. 
We can define a natural semidirect product $H\rtimes G/H$  
as $\{(h, gH)|h\in H, gH\in G/H\}$ with the action $h^{gH}=h^g$.
Since $H$ is abelian, this semidirect product is well-defined.

\begin{lem}
\label{bingo1}
    Assume that a Sylow $p$-subgroup of $G$ is abelian
    and $H$ is a normal $p$-subgroup of $G$. 
    Then $\N(G)\subseteq \N(H\rtimes G/H)$.
\end{lem}

\begin{proof}
    Let $\widetilde{G}=H\rtimes G/H$. 
    It is clear that $|G|=|\widetilde{G}|$.
    We will show that for any element $g\in G$,
    there exists an element $\widetilde{g}$ from $\widetilde{G}$ 
    such that $|\C_G(g)|=|\C_{\widetilde{G}}(\widetilde{g})|$.

    If $g$ is a $p'$-element in $G$,
    by Lemma \ref{basic}(iv), we have $\C_{G/H}(gH) = \C_G(g)H/H$.
    Hence $|\C_{G/H}(gH)| = |\C_G(g)| / |\C_G(g)\cap H|
    = |\C_G(g)|/|\C_H(g)|$.
    Consider the element $(1, gH)$ in $\widetilde{G}$.
    Let $(a,bH)\in C_{\widetilde{G}}((1,gH))$.
    By calculation, we have $a\in \C_H(g)$ and $bH\in \C_{G/H}(gH)$.
    Therefore, 
    $|C_{\widetilde{G}}((1,gH))|= |\C_H(g)|\times |\C_{G/H}(gH)| = |C_G(g)|$.

    Assume that $g$ is a $p$-element in $G$.
    If $g\in H$, It is easy to verify that $|\C_G(g)|= |\C_{\widetilde{G}}((g,H))|$.
    So it suffices to consider when $g\notin H$.
    Let $T\leq G$ be such that $T/H = \C_{G/H}(gH)$.
    If $T=\C_G(g)$, then we can get $|C_{\widetilde{G}}((1,gH))|= |C_G(g)|$ as above.
    Consider $T>\C_G(g)$. 
    By Lemma \ref{L4}, 
    we have $g = xy$, where $1\neq x\in \C_G(T)$, 
    $1\neq y\in H$ and $\C_G(y)\cap T=\C_G(g)$.
    Consider $(y, gH)\in \widetilde{G}$.
    It is easy to verify that $|\C_G(g)|=|\C_{\widetilde{G}}((y,gH))|$.

    If $g$ is neither a $p$-element nor a $p'$-element, 
    then $g = uv$ with $uv=vu$, 
    where $|u| = |g|_p$ and $|v| = |g|_{p'}$. 
    Moreover, $C_G(g) = C_G(u) \cap C_G(v)$.  
    If $u\in H$, then $(u, vH)=(u,H)(1, vH)=(1,vH)(u,H)$.
    We have $\C_{\wtG}((u, vH))=\C_{\wtG}((u,H))\cap \C_{\wtG}((1,vH))$.
    Let $(a,bH)\in \C_{\widetilde{G}}((u,vH))$.
    By calculation, we have $a\in \C_H(v) ,bH\in \C_G(u)/H \cap \C_{G/H}(vH)$.
    Then 
    $$\begin{aligned}
        |\C_{\wtG}((u,vH))| 
        &=|\C_H(v)||\C_G(u)/H \cap \C_{G/H}(vH)| =|\C_H(v)||\C_G(u)/H \cap \C_{G}(v)H/H| \\
        &=|(\C_H(v)||(\C_G(u)\cap \C_{G}(v)H)/H| = |(\C_H(v)||\C_G(u)\cap \C_{G}(v))H/H| \\
        &=|(\C_H(v)||\C_G(g)H/H|
        =|\C_G(g)|.
    \end{aligned}$$
    If $u\notin H$ and $\C_G(u)/H=\C_{G/H}(uH)$, 
    then 
    $$\begin{aligned}
        \C_{G/H}(gH) &=\C_{G/H}(uH)\cap \C_{G/H}(vH)=\C_G(u)/H\cap \C_G(v)H/H \\
                     &=(\C_G(u)\cap \C_G(v)H)/H =(\C_G(u)\cap \C_G(v))H/H=\C_G(g)H/H.
    \end{aligned}$$
    We can get $|\C_{\wtG}((1,gH))=|\C_G(g)|$ as above.
    
    Finally, we consider the case where $u\notin H$ and $\C_G(u)/H< \C_{G/H}(uH)$.
    Let $T_u \leq G$ be such that $T_u/H = \C_{G/H}(uH)$.
    By Lemma \ref{L4}, we have $u=x_u y_u$ where $1\neq x_u\in \C_G(T_u)$, 
    $1\neq y_u\in H$ and $\C_G(y_u)\cap T_u=\C_G(u)$. 
    Since $v\in \C_G(u)$, we have $v\in \C_G(y_u)$.
    Hence $(y_u, uvH)=(y_u, uH)(1, vH)=(1, vH)(y_u, uH)$.
    Similarly, we can verify that $|\C_{\wtG}((y_u, gH))|=|\C_G(g)|$.
    The lemma is proved.
\end{proof}

\begin{lem}
\label{bingo2}
    Assume that a Sylow $p$-subgroup of $G$ is abelian
    and $H$ is a normal $p$-subgroup of $G$. 
    Then $ \N(H\rtimes G/H)\subseteq \N(G)$.
\end{lem}

\begin{proof}
    Let $\wtG = H\rtimes G/H$ and $(h,gH)\in \wtG$. 
    We will show that there exists an element $g'\in G$ such that
    $|\C_{\wtG}((h,gH))|=|\C_G(g')|$.

    If $(h,gH)$ is a $p$-element, then $g$ must be a $p$-element.
    By Lemma \ref{L4}, without loss of generality, 
    we let $g$ satisfy $\C_{G/H}(gH)=\C_G(g)/H$.
    By Lemma \ref{centre}, we have $\C_G(gh) = \C_G(g)\cap \C_G(h)$.
    Let $(a,bH)\in \C_{\wtG}((h,gH))$.
    We can get that $a\in H$, $b\in \C_G(g)\cap \C_G(h)=\C_G(gh)$.
    Hence, $|\C_{\wtG}((h,gH))|=|H||\C_G(gh)/H|=|\C_G(gh)|$.

    If $(h,gH)$ is a $p'$-element, without loss of generality, 
    we can let $g$ be a $p'$-element.
    By Lemma \ref{size}, $|\C_{\wtG}((h,gH))|=|\C_{\wtG}((1,gH))|$.
    It easy to verify that $|\C_{\wtG}((1,gH))|=|\C_G(g)|$.

    If $(h, gH)$ is neither a $p$-element nor a $p'$-element,
    it can be written as the product of a $p$-element and a $p'$-element that commute with each other.
    By Lemmas \ref{size} and \ref{L4}, without loss of generality, 
    let $(h,gH)=(h, nH)(1, kH)=(1, kH)(h, nH)$, 
    where $n$ is a $p$-element such that $\C_{G/H}(nH) = \C_G(n)/H$
    and $k$ is a $p'$-element such that $k\in \C_G(h)\cap \C_G(n)=\C_G(hn)$.
    Let $(a, bH)\in \C_{\wtG}((h,gH))$.
    we have $(a,bH)\in \C_{\wtG}((h, nH))\cap \C_{\wtG}((1,kH))$.
    By calculation, we can get that $a\in \C_H(k)$ 
    and $bH\in \C_G(hn)/H \cap \C_G(k)H/H=\C_G(hnk)H/H$.
    Hence 
    $|\C_{\wtG}((h,gH))|
    =|\C_H(k)||\C_G(hnk)H/H|
    =|\C_G(hnk)|$.
\end{proof}

The following lemma is the key to the entire proof
and also quite interesting on its own.

\begin{lem}
\label{bingo}
    Assume that a Sylow $p$-subgroup of $G$ is abelian
    and $H$ is a normal $p$-subgroup of $G$. 
    Then $\N(G)= \N(H\rtimes G/H)$.
\end{lem}

\begin{proof}
    It follows immediately from Lemmas \ref{bingo1} and \ref{bingo2}.
\end{proof}

\begin{lem}
\label{key}
    Let $G$ be an $A$-group. Then $\N(G)=\N(\F(G)\rtimes G/\F(G))$.
\end{lem}

\begin{proof}
    Let $H$ and $N$ be abelian normal subgroups of $G$ and $H\cap N=1$.
    Let $G_1 = H\rtimes G/H$ and $N_1 = \{(1,gH)|g\in N\}$.
    It is clear that $N_1\cong NH/H \cong N$ is also an abelian normal subgroup of $G_1$.
    Let $G_2 = N_1\rtimes G_1/N_1$ and $G_3 = HN\rtimes G/HN$.
    We will show that $G_3 \cong G_2$.

    Let $(hn,gHN)\in G_3$, where $h\in H, n\in N, g\in G$.
    Let $\phi$ be a map from $G_3$ to $G_2$ 
    and $\phi((hn, gHN)) = ((1, nH), (h,gH)N_1)$.
    Since $(h,gH)N_1=\{(h,gH)(1, kH)|k\in N\}
    =\{(h, gkH)|k\in N\}$,
    $\phi$ is well-defined.
    We have 
    $$\begin{aligned}
        \phi(h_1 n_1, g_1 HN)\phi(h_2 n_2, g_2 HN) 
        &=((1, n_1 H), (h_1,g_1 H)N_1)((1, n_2 H), (h_2,g_2 H)N_1) \\
        &=((1, n_1 H)(1,n_2H)^{(h_1, g_1 H)^{-1}}, (h_1,g_1 H)(h_2,g_2 H)N_1) \\
        &=((1, n_1 n_2^{g_1^{-1}}H), (h_1 h_2^{g_1^{-1}}, g_1g_2H)N_1) \\
        &=\phi((h_1 h_2^{g_1^{-1}} n_1 n_2^{g_1^{-1}}, g_1 g_2 HN)) \\
        &=\phi((h_1 n_1, g_1 HN)(h_2 n_2, g_2 HN))
    \end{aligned}.$$
    Hence $\phi$ is a homomorphism.
    It is easy to verify that the kernel of $\phi$ is trivial.
    Therefore $\phi$ is an automorphism and $G_3\cong G_2$.

    Let $\pi(\F(G))=\{p_1, p_2,\ldots, p_r\}$. 
    We have $\F(G)=\rmO_{p_1}(G)\times \rmO_{p_2}(G)\times\cdots\times\rmO_{p_r}(G)$.
    Since $G$ is an $A$-group, 
    $\rmO_{p_i}(G)$ is abelian for every $1\leq i\leq r$.
    By repeatedly applying the above result and Lemma \ref{bingo}, 
    we can get $\N(G)=\N(\F(G)\rtimes G/\F(G))$.
\end{proof}

\begin{lem}
    Let $G$ be an $A$-group. 
    Then $G$ is an abelian group 
    if and only if $\F(G)\rtimes G/\F(G)$ is abelian.
\end{lem}

\begin{proof}
    It is enough to verify the sufficiency.
    Suppose that $\F(G)\rtimes G/\F(G)$ is abelian.
    Then $G$ is solvable.
    Hence $\C_G(\F(G))\leq \F(G)$.
    On the other hand, 
    from the construction of $\F(G) \rtimes G/\F(G)$, 
    we have $\F(G) \leq \Z(G)$.
    Therefore $G$ is abelian.
\end{proof}

\begin{lem}
\label{CA}
    Let $G$ be an $A$-group and $G=F\rtimes T$, 
    where $F$ is the Fitting subgroup of $G$.
    Then 

    (i) For $y\in T$, we have $F=\C_F(y)\times [F,y]$.

    (ii) Let $g\in G$. Then there exists $k\in G$
    such that $g^k=xy$ with $x\in F$, $y\in T$
    and $[x,y]=1$. 
    Here either $x$ or $y$ can be the identity.

    (iii) Let $x\in F$, $y\in T$ with $xy=yx$.
    Then $\C_G(xy)=\C_G(x)\cap \C_G(y)$.
    Moreover, if $\C_G(x)\C_G(y)=G$,
    then $\Ind(G,xy)=\Ind(G,x)\Ind(G,y)$.
\end{lem}

\begin{proof}
    (i) Let $p\in \pi(F)$ and $y=st$, 
    where $s$ is a $p$-element and $t$ is a $p'$-element
    and $st=ts$.
    Since $G$ is an $A$-group,
    the action of $y$ on $\rmO_p(G)$ is given by the action of $t$.
    By Lemma \ref{Go},
    we have $\rmO_p(G)
    =\C_{\rmO_p(G)}(\gen{y})\times [\rmO_p(G),\gen{y}]
    =\C_{\rmO_p(G)}(y)\times [\rmO_p(G),y]$.
    Since $F$ is abelian, (i) holds.

    (ii) We note that $g$ can be written as $wy$ 
    with $w\in F$ and $y\in T$.
    By (i), $w=uv$ with $u\in \C_F(y)$ and $v\in [F,y]$.
    Note that $vy=y^{k^{-1}}$ for some $k\in F$.
    Then $g=wy=uvy= u y^{k^{-1}}$ and $g^k=u^k y=uy$.
    Here $u$ is the required $x$.

    (iii) 
    Since $\C_G(xy)F/F\leq \C_{G/F}(yF)=\C_T(y)F/F$,
    we have $\C_G(xy)\leq F\rtimes \C_T(y)$.
    Let $h = ba \in \C_G(xy)$, where $a\in F$ and $b\in \C_T(g)$.
    By (i), $F=\C_F(y)\times [F,y]$.
    Since $(xy)^{ba} = x^b y^a = xy$, 
    we have $x^{-1}x^b = y (y^a)^{-1} = [a^{y^{-1}}, y] \in [F,y]$.
    Since $x,b\in \C_G(y)$, $x^{-1}x^b\in \C_F(y)$.
    Therefore $x^{-1}x^b = 1$, which implies that $h\in \C_G(x)$.
    Consequently, $h\in \C_G(y)$.
    Thus, $\C_G(xy) \leq \C_G(x)\cap \C_G(y)$.
    It is clear that $\C_G(x)\cap \C_G(y)\leq \C_G(xy)$.
    Therefore $\C_G(xy)=\C_G(x)\cap \C_G(y)$.

    If $\C_G(x)\C_G(y)=G$, we have
    $$ 
        \Ind(G,x)\Ind(G,y)
        =\dfrac{|G|}{|\C_G(x)|}\dfrac{|G|}{|\C_G(y)|} 
        =\dfrac{|G|^2}{|\C_G(x)\C_G(y)||\C_G(x)\cap \C_G(y)|}
        =\dfrac{|G|}{|\C_G(xy)|}
        =\Ind(G,xy).
    $$
\end{proof}

\section{Proof of Theorem}

Let $G$ satisfy the conditions in the theorem, 
and let $F$ be the Fitting subgroup of $G$. 
By \cite[Theorem 1]{Camina2006}, 
we know that $G$ is solvable.
Hence $F>1$ and $\C_G(F)=F$.

\begin{lem}
\label{CC}
    $F$ contains an element $g$ such that $\Ind(G,g)_p=|G/F|_p$.
\end{lem}

\begin{proof}
    Let $P\in \Syl_p(G)$.
    Assume that $F$ is a $p$-group.
    Since $P$ is abelian and $\C_G(F)=F$, we have $F=P$.
    For any $g\in F$, we have $\Ind(G,g)_p=|G/F|_p=1$.
    
    Now we consider when $F$ is not a $p$-group.
    Let $q\in \pi(F)$ and $q\neq p$. 
    Consider the action of $P\C_G(\rmO_qG))/\C_G(\rmO_qG))$ on $\rmO_q(G)$.
    Denote $P\C_G(O_q(G))$ by $K$.
    By Lemma \ref{CL2}, 
    there exists an element $x_q\in \rmO_q(G)$ such that 
    $\Ind_K(x_q) = |K/\C_G(\rmO_qG))|$.
    Hence $\C_K(x_q) = \C_G(\rmO_q(G))$, 
    which implies that $\C_P(x_q) = \C_P(\rmO_q(G))$.
    Let $\pi(F)\setminus \{p\}= \{q_1, q_2, \ldots, q_r\}$.
    For each $q_i$, 
    we can similarly choose an element in $\rmO_{q_i}(G)$ as above, denoted by $x_i$.
    Let $x= \Pi _{i=1}^{r} x_i$.
    We have $\C_P(x)=\C_G(x)\cap P=\C_G(x_1)\cap \ldots \cap \C_G(x_r) \cap P
    =\C_P(x_1)\cap \ldots \cap \C_P(x_r)=\C_P(\rmO_{q_1}(G))\cap \ldots \cap\C_P(\rmO_{q_r}(G))
    =\rmO_p(G)$.
    The last equality holds due to the solvability of $G$, which implies that $F = \C_G(F)$.
    Hence $\Ind(G,x)_p=|G/F|_p$.
\end{proof}

Let $G$ be a minimal counterexample to the statement of the theorem. 
Put $F=\F(G)$, $F_2=\F_2(G)$. 
Here $\F_2(G)$ is the second Fitting subgroup of $G$,
that is, $\F_2(G)/\F(G)=\F(G/\F(G))$.

\begin{lem}
\label{perfect}
If $\Ind(F_2,g)=p^x$, then $\Ind(G/F,gF)_p=1$.
\end{lem}

\begin{proof}
    Suppose that this lemma is false. 
    Let $g\in G$ be such that $\Ind(F_2,g) = p^x$ and $\Ind(G/F,gF)_p>1$. 
    By Lemma \ref{basic}(ii), 
    we can assume that $g$ is an element of primary order. 
    If $g$ acts on $F$ trivially, then $g\in F$, 
    which contradicts that $\Ind(G/F,gF)_p>1$. 
    Therefore, we assume that $g$ acts nontrivially on $F$.

    Let $\pi(|g|)=\{t\}$ and
    $\overline{\phantom{x}}: G\rightarrow G/F$ be a natural homomorphism. 
    We have $\Ind(\overline{G},\overline{g})_p>1$.
    Let $T\in \Syl_t(\overline{F_2})$,
    $R\in \Syl_p(\overline{G})$, 
    $P\in \Syl_p(F)$. 
    Note that $R\leq \N_{\overline{G}}(T)$. 
    Let $H=\langle \overline g^R\rangle$. 
    Since $R$ acts non-trivially on $\overline{g}$, 
    $R$ acts non-trivially on $H$. 
    Since $P$ is abelian,
    we can construct the following group as a natural semidirect product:
    $K=(P\rtimes H)\rtimes R$.
    We have $R\leq C_K(P)$ and so $\C_K(P)=P\C_H(P)R$.
    Since $\C_K(P)\teq K$, we have $\C_H(P)R\teq HR$.
    Moreover, $H\teq HR$ and $H/\C_H(P)\cap R\C_H(P)/\C_H(P)=1$.
    Therefore $HR/\C_{H}(P)$ is abelian and so $[HR, HR]\leq \C_H(P)$.

    By Lemma \ref{Go}, we have $H=\C_H(R)\times[H,R]$. 
    We have $[H,R]\leq [HR, HR]\leq \C_H(P)$. 
    Since $\overline{g}$ acts trivially on $\rmO_{p'}(F)$, 
    then $H$ acts trivially on $\rmO_{p'}(F)$. 
    Thus $[H,R]$ acts trivially on $\rmO_{p'}(F)$. 
    Consequently, $[H,R]$ acts trivially on $F$. 
    Since $G$ is solvable, $\C_G(F)=F$ and so $[H,R]=1$. 
    Therefore $H=\C_H(R)$, 
    which contradicts that $R$ acts nontrivially on $\overline{g}$.
    Hence $\Ind(G/F, gF)=1$.
\end{proof}

By Lemma \ref{key}, we can assume that our group $G$ has a trivial center and can be represented as $F\rtimes T$, where $F=F(G)$.

\begin{lem}
\label{primeiro}
    There exists $f\in F$ such that $\Ind(G,f)=|T|_p$.
\end{lem}

\begin{proof}
    Let $x\in G$ be such that $\Ind(G,x)=|G||_p$.
    By Lemma \ref{CA}(ii), 
    we can assume that $x=ab$, where $a\in F$, $b\in T$ and $ab=ba$.
    By Lemma \ref{basic}(i), $\Ind(F_2, x)$ is also a power of $p$.
    By Lemma \ref{perfect}, we have $\Ind(G/F, xF)=1$.
    Then $\Ind(G/F, bF)=\Ind(G/F, xF)=1$, which implies that $\Ind(T,b)=1$.
    If follows that $\C_G(a)\C_G(b)=G$.
    From Lemma \ref{CA}(iii), we have $|G||_p=\Ind(G, x)=\Ind(G,a)\Ind(G,b)$.
    
    By Lemma \ref{CC}, $F$ contains an element $d$ such that $\Ind(G,d)_p=|T|_p$. 
    We can assume that $d$ is a $p'$-element and $d\in \C_F(b)$. 
    Clearly, $\C_G(d)\C_G(b)=G$. 
    By Lemma \ref{CA}(iii), we have $\Ind(G, db)=\Ind(G,d)\Ind(G,b)$.
    Hence $\Ind(G, db)_p=|T|_p\Ind(G,b)$.
    This forces $\Ind(G,a)=|T|_p$.
\end{proof}

\begin{lem}
\label{segundo}
    Let $g\in G$ be such that $\Ind(F,g)_p$ is maximal. 
    Then there exists an element $g'\in T$ such that $\Ind(G,g')=\Ind(F,g)_p$. 
    Moreover, $|G||_p=|T|_p\Ind(F,g)_p$.
\end{lem}

\begin{proof}
    It follows from Lemma \ref{primeiro} that 
    there exists $f\in F$ such that $\Ind(G,f)=|T|_p$,
    which implies that $\C_G(f)$ contains a Hall $p'$-subgroup of $G$. 
    Since $F$ is abelian, we can assume that $g\in T$.
    Since a Sylow $p$-subgroup of $G$ is abelian, 
    we can assume that $g$ is a $p'$-element such that $g\in \C_G(f)$. 
    By Lemma \ref{CA}(iii), $\C_G(fg)=\C_G(f)\cap \C_G(g)$.
    Let $P\in \Syl_p(\C_G(fg))$.
    Since $\Ind(G,f)=|T|_p$, every $p$-element in $T$ is not contained in $\C_G(f)$.
    Hence $P\leq F$ and so $P\leq \C_G(gf)\cap F=\C_F(g)$.
    We have $\Ind(G,fg)_p=|G/\C_G(fg)|_p=|G|_p/|P|=|T|_p |F|_p/|P|=|T|_p\Ind(F,g)_p$.

    Let $y\in G$ be such that $\Ind(G,y)=|G||_p$. 
    By Lemma \ref{CA}(ii), 
    we can assume that $y=cd$, 
    where $c\in F$, $d\in T$ and $cd = dc$. 
    As we showed in the proof of Lemma \ref{primeiro}, 
    $\Ind(T,d)=1$. 
    Therefore $\C_G(c)\C_G(d)=G$.
    By Lemma \ref{CA}(iii), $\Ind(G,y)=\Ind(G,c)\Ind(G,d)=\Ind(G,c)\Ind(F,d)$.
    By the definition of $|G||_p$, 
    it must hold that $\Ind(G,d)=\Ind(F,g)_p$.
    It follows that $|G||_p=|T|_p\Ind(F,g)_p$.
    Thus Lemma \ref{segundo} holds.

\end{proof}

\textbf{Proof of the theorem.}
    We know that $|G||\in \N(G)$ and $|G||$ is the largest element in $\N(G)$.
    Let $h\in G$ such that $\Ind(G,h)=|G||$. 
    By Lemma \ref{CA}(ii), 
    we can assume that $h=ab$, where $a\in F$, $b\in T$ and $ab=ba$.
    If $b=1$, from Lemma \ref{segundo} we can get that $T\leq \C_G(F)$, a contradiction.
    Hence $b\neq 1$.
    Let $p$ be a prime divisor of the order of $b$. 
    We have $\Ind(G, ab)_p=|G|_p/|\C_G(ab)|_p=|T|_p|F|_p/|\C_G(ab)|_p
    \leq |T|_p|F|_p/|\C_F(b)|_p=|T|_p\Ind(F,b)_p.$
    By Lemma \ref{segundo},
    $\Ind(F,b)_p$ must be maximal and $\Ind(G, ab)_p=|T|_p\Ind(F,b)_p.$
    Hence $|\C_G(ab)|_p=|\C_F(b)|_p$. 
    It follows that every $p$-element of $\C_G(ab)$ must be contained in $F$.
    However, the $p$-part of $b$ is contained in $\C_G(ab)\cap T$.
    This contradiction completes the proof.
    


\begin{thebibliography}{s2}

\bibitem{Camina2011} A. R. Camina, R. D. Camina, 
    The influence of conjugacy class sizes on the structure of finite groups: a survey, 
    \textit{Asian-European Journal of Mathematics}, 
    \textbf{4}(4) (2011), 559--588.

\bibitem{Camina2006}
    A. R. Camina, R. D. Camina,
    Recognising nilpotent groups,
    \textit{Journal of Algebra}, 
    \textbf{300}(1) (2006), 16--24.


\bibitem{Casolo2012}
    C. Casolo, E. M. Tombari,
    Conjugacy class sizes of certain direct products,
    \textit{Bulletin of the Australian Mathematical Society}, 
    \textbf{85}(2) (2012), 217--231.

\bibitem{Cossey2000}
    J. Cossey, T. Hawkes,
    Sets of $p$-powers as conjugacy class sizes,
    \textit{Proceedings of the American Mathematical Society}, 
    \textbf{128}(1) (2000), 49--51.

\bibitem{Gorenstein}
    D. Gorenstein,
    Finite Groups,
    \textit{Harper’s Series in Modern Mathematics, Harper} \& \textit{Row}, 
    New York, 1968.

\bibitem{GorA2}
    I. B. Gorshkov,
    On Thompson's conjecture for alternating and symmetric groups of degree more than 1361,
    \textit{Proceedings of the Steklov Institute of Mathematics}, 
    \textbf{293}(1) (2016), 58--65.


\bibitem{Gorshkov2025} 
    I. B. Gorshkov, C. Shao, T. M. Mudziiri Shumba,
    Description of direct products by sizes of conjugacy classes,
    \textit{Communications in Algebra}, \textbf{53}(6) (2025), 2292--2298.

\bibitem{Ito1953} N. It\^o, 
    On finite groups with given conjugate types \Romannum{1}, 
    \textit{Nagoya Mathematical Journal}, \textbf{6} (1953), 17--28.

\bibitem{Ishikawa2002} K. Ishikawa, 
    On finite $p$-groups which have only two conjugacy lengths, 
    \textit{Israel Journal of Mathematics}, \textbf{129} (2002), 119--123.


\end{thebibliography}
\end{document}